\documentclass{article}
\usepackage{mystyle}
\usepackage[cp1251]{inputenc}
\usepackage{amssymb}
\usepackage{amsmath}
\usepackage{latexsym}
\usepackage{amsthm}
\newtheorem {theorem}{Theorem}
\newtheorem {definition}{Definition}
\newtheorem {proposition}{Proposition}
\newtheorem {lemma}{Lemma}

\theoremstyle{remark}
\newtheorem*{remark}{Remark}

\DeclareMathOperator{\SSS}{\mathbb{S}^2}
\DeclareMathOperator{\RP}{\mathbb{RP}^2}

\begin{document}
\title{Upper bounds for the first eigenvalue of the Laplacian on non-orientable surfaces}
\author{Mikhail A. Karpukhin}
\date{} 
\maketitle
\begin{abstract}
In 1980 Yang and Yau~\cite{YY} proved the celebrated upper bound for the first eigenvalue on an orientable surface of genus $\gamma$. Later Li and Yau~\cite{LY} gave a simple proof of this bound by introducing the concept of conformal volume of a Riemannian manifold. In the same paper they proposed an approach for obtaining a similar estimate for non-orientable surfaces. In the present paper we formalize their argument and improve the bounds stated in~\cite{LY}.
\end{abstract}

\section{Introduction and main results}

Let $M$ be a closed surface and $g$ be a Riemannian metric on $M$. Let us consider the associated Laplace-Beltrami operator
$\Delta$ acting on the space of smooth functions on~$M$,
$$
\Delta f = -\frac{1}{\sqrt{|g|}}\frac{\partial}{\partial x^i}\left(\sqrt{|g|}g^{ij}\frac{\partial f}{\partial x^j}\right).
$$
It is well-known that the spectrum of $\Delta$ is non-negative and consists only of eigenvalues, each eigenvalue has a finite
 multiplicity and the eigenfunctions are smooth. Let us denote the eigenvalues of $\Delta$ by 
$$
0 = \lambda_0(M,g) < \lambda_1(M,g) \leqslant \lambda_2(M,g) \leqslant \lambda_3(M,g) \leqslant \ldots,
$$
where eigenvalues are written with multiplicities.

Let us fix $M$ and consider eigenvalues as functionals on the space of all Riemannian metrics. These functionals possess the following rescaling property,
$$
\lambda_i(M,tg) = \frac{\lambda_i(M,g)}{t}.
$$
Thus, in order to get scale-invariant functionals on the space of Riemannian metrics one has to normalize eigenvalues. The most natural way to do it is to multiply by the area, so we consider the following functionals
$$
\Lambda_i(M,g) = Area_g(M)\lambda_i(M,g).
$$
In the present paper we are concerned only with the $\Lambda_1(M,g)$. 
\begin{remark}
In further exposition we use the notation $\Lambda_1(\Omega,g)$, where $\Omega$ is a manifold with boundary. In this case, the Neumann boundary conditions on $\partial\Omega$ are implied and everything stated above remains true. In general a Riemannian surface is assumed to be a closed Riemannian surface unless stated otherwise. 
\end{remark}

In 1980, Yang and Yau~\cite{YY} proved that for an orientable surface of genus $\gamma$ the following inequality holds:
\begin{equation}
\Lambda_1(M,g) \leqslant 8\pi\left[\frac{\gamma+3}{2}\right].
\label{YY}
\end{equation}
  
Later Li and Yau simplified the proof of that inequality in the paper~\cite{LY} using the concept of conformal volume. They have also outlined the proof of a similar inequality for non-orientable surfaces. Their inequality reads 
\begin{equation}
\Lambda_1(M,g) \leqslant 24\pi\left[\frac{\gamma+3}{2}\right].
\label{LY}
\end{equation}
Here the genus $\gamma$ of a non-orientable surface is the genus of its orientable double cover, so that $\chi(M) = 1-\gamma$.
However, their argument contained a miscalculation, yielding the existence of a map of even degree from a non-orientable surface $M$ to a projective space, which contradicted Hurwitz identities (see e.g.~\cite{EKS}) in case the genus of $M$ was even. In the present paper we modify their ideas and obtain an improved upper bound.
\begin{theorem}
\label{MainTheorem}
Let $(M,h)$ be a non-orientable Riemannian surface of genus $\gamma$. Then one has the inequality
$$
\Lambda_1(M,h)\leqslant 16\pi \left[\frac{\gamma+3}{2}\right].
$$ 
\end{theorem}

Moreover we are able to say more if the genus $\gamma$ equals $2$.

\begin{theorem}
\label{gamma2}
Let $(M,h)$ be a non-orientable surface of genus $2$. Then one has the inequality
$$
\Lambda_1(M,h)\leqslant 16\pi
$$
\end{theorem}

An interesting question of spectral geometry is determining the value $\sup_g \Lambda_1(M,g)$ for a fixed surface $M$.
Thus inequality~(\ref{YY}) guarantees that this value is finite for orientable surfaces while Theorem~\ref{MainTheorem} asserts the same for non-orientable surfaces. It is known that for a sphere supremum is equal to $8\pi$ and is achieved on a round metric, see~\cite{Hersch}. Recently, the value of this quantity has been found for a projective plane (see~\cite{LY}), a torus (see~\cite{Nadirashvili}) and a Klein bottle (see~\cite{EGJ,JNP}). Up to a computer calculation the result is known for an orientable surface of genus $2$ (see~\cite{JLNNP}). Let us also mention a series of works on extremal metrics for functionals $\Lambda_i(M,g)$ on torus and Klein bottle~\cite{Karpukhin1,Karpukhin2,Lapointe,Penskoi1,Penskoi2,Penskoi3,PenskoiSurvey}.

The paper is organized as follows. In Section~\ref{ConfVol} we give a defnition of conformal volume and outline some of its properties. In Section~\ref{prelim} we explain the main idea of the proof of Theorem~\ref{MainTheorem}. Section~\ref{algem} contains some facts about the geometry of real curves that are necessary for the proof. Finally in Sections~\ref{final} and~\ref{hyperelliptic} we prove Theorems~\ref{MainTheorem} and~\ref{gamma2}.

\section{Conformal volume}

In the paper~\cite{LY} the authors introduced the concept of conformal volume $V_c(M,g)$ for any Riemannian surface $(M,g)$, possibly with boundary. It is defined by the following min-max formula
$$
V_c(M,g) = \lim_{n\to\infty}\inf_{\phi}\sup_{h\in G_n} \int_M dV_h,
$$
where $\phi$ ranges over all possible conformal mappings $M\to\mathbb{S}^n$, $G_n$ is a group of conformal diffeomorphisms of $\mathbb{S}^n$ and $dV_h$ is a volume element associated to the metric (possibly degenerate) $\phi^*h^*g_0$, where 
$g_0$ is the canonical metric on the sphere.
The next theorem is a compilation of results from~\cite{LY}  that allow us to immediately present a proof of inequality~(\ref{YY}) given in~\cite{LY}. 
\label{ConfVol}

\begin{theorem}[Li and Yau~\cite{LY}]
The conformal volume $V_c(M,g)$ possesses the following properties
\label{ThConf}
\begin{itemize}
\item[1)]If $N \to M$ is a conformal map of degree $d$ then $V_c(N,g)\leqslant d V_c(M,g)$.
\item[2)] One has the inequality
$$
\Lambda_1(M,g)\leqslant 2V_c(M,g)
$$ 
\item[3)] If $\Omega\subset M$ is a domain of $M$ then $V_c(\Omega,g)\leqslant V_c(M,g)$. As a result $\Lambda_1(\Omega,g)\leqslant 2V_c(M,g)$.
\item[4)] $V_c(\mathbb{S}^2,g) = 4\pi$, $V_c(\mathbb{RP}^2,h) = 6\pi$ for any choice of metrics $g$ and $h$.
\end{itemize}
\end{theorem}
\begin{remark}
Let us remind the reader that if $\partial M$ is not empty then $\lambda_1(M,g)$ denotes the first non-trivial Neumann eigenvalue of Laplacian on $M$.
\end{remark}

Using this theorem it is easy to prove inequlaity~(\ref{YY}). Indeed it is a known fact (see, e.g.~\cite{GH}) that for any orientable Riemannian surface of genus $\gamma$ there is a conformal map to the sphere of degree at most $\left[\frac{\gamma+3}{2}\right]$. Then combination of items $2$ and $4$ from the theorem completes the proof.

It is natural to try to do the same for a non-orientable surface, where one replaces the sphere $\SSS$ with the projective space $\mathbb{RP}^2$. This is an approach proposed in~\cite{LY}. However as we discussed in the introduction the argument presented there contains an error. Moreover, we were not able to find in the literature an answer to the question whether there exists a conformal map $M\to \mathbb{RP}^2$ for a given non-orientable surface $M$.

It turns out that it is more convenient to use a disk instead of $\RP$ in the aforementioned argument. The reasoning for that comes from a theory of real curves and will be explained 
below in Section~\ref{algem}.   

\section{Preliminaries}

Let $\pi\colon\tilde M\to M$ be a double cover of $M$, where $\tilde M$ is an orientable surface of genus $\gamma$. We endow $\tilde M$ with the metric $g = \pi^*h$ and the corresponding conformal structure. Then $\pi$ can be seen as a factorization of $\tilde M$ by a free isometric involution $\sigma\colon\tilde M\to\tilde M$ that acts by exchanging the sheets of the cover $\pi$.    
\label{prelim}

Since $\sigma$ is an isometry, one has $\sigma^*\Delta = \Delta\sigma^*$ and therefore there is a common basis of eigenfunction for these operators. In particular, eigenvalues of the Laplacian are divided into odd and even, i.~e. those corresponding to odd and even eigenfunctions respectively. Let us denote by $\lambda_1^{\sigma}(\tilde M,g)$ the first even non-trivial eigenvalue of $\Delta$ on $\tilde M$. Then one observes that $\lambda_1^{\sigma}(\tilde M,g) = \lambda_1(M,h)$.

Assume $f\colon\tilde M\to\SSS$ is a conformally branched covering of degree $d$. Yang and Yau in their paper introduced a procedure of pushing forward a metric $g$ on $\tilde M$ to a metric $g^*$ on $\SSS$ with isolated conical singularities at branch points of $f$. 
\begin{definition}
\label{push}
Let $f$ be a conformally branched covering $f\colon\tilde M\to\SSS$.
Then preimage of a suitable neighbourhood $U$ of a non-branching point $p\in\SSS$ is a collection of $d$ disks $U_i$, $i =1,\ldots,d$ and the restriction $f_i:=f|_{U_i}$ to each disk is a diffeomorphism. The push-forward metric $g^*$ on $U$ is defined as a sum $\Sigma (f_i^{-1})^*g$.
\end{definition}
 In their paper Yang and Yau proved the following theorem for the metric $g^*$.
\begin{theorem}
For any function $u\in C^\infty(\SSS)$ one has the following
\begin{equation}
\label{YY1}
\int_{\SSS} u \, dv_{g^*} = \int_{\tilde M} (f^*u) dv_{g}
\end{equation}
and
\begin{equation}
\label{YY2}
d \int_{\SSS} |\nabla_{g^*} u|^2 dv_{g^*} = \int_{\tilde M} |\nabla_{g} (f^*u)|^2 dv_{g}.
\end{equation}
\end{theorem}
Yang and Yau used this theorem to prove an upper bound for $\lambda_1(\tilde M,g)$. Here we adapt their approach to prove an upper bound for $\lambda_1^{\sigma}(\tilde M,g)$. 

Suppose that there is a conformal involution $\tau$ of a sphere such that $f$ satisfies the following relation
\begin{equation}
\label{tau}
 f(\sigma(x)) = \tau (f(x)).
 \end{equation}
 
\begin{lemma} 
The transformation $\tau$ is an isometry of $(\SSS,g^*)$.
\end{lemma}
\begin{proof}
Let $f$, $U$, $U_i$ and $f_i$ be as in definition~\ref{push}, i.e. $\pi^{-1}(U) = \coprod U_i$. Applying $\sigma$ to this equality and using relation~(\ref{tau}) we obtain 
$\pi^{-1}(\tau(U)) = \coprod \sigma(U_i)$. Let $\tilde f_i$ be restrictions of $f$ on $\sigma(U_i)$. Then 
$$
(\tilde f_i^{-1})^*g = (\tilde f_i^{-1})^*\sigma^*g = (\tilde f_i^{-1}\circ\sigma)^*g = (\tau\circ f_i^{-1})^*g = \tau^*(f_i^{-1})^*g.
$$
Summing these relations over $i$ we obtain $\tau^*g^* = g^*$.
\end{proof}

 Thus we can consider the first $\tau$-even eigenfunction $\phi$ corresponding to $\lambda_1^{\tau}(\SSS,g^*)$. The function $f^*\phi$ is obviously $\sigma$-even and by~(\ref{YY1}) the mean of $f^*\phi$ over $\tilde M$ is $0$. Therefore we can use $f^*\phi$ as a test function in the variational characterization of $\lambda_1^{\tau}(\tilde M,g)$. Application of~(\ref{YY2}) yields
$$
\lambda_1^{\tau}(\tilde M,g) \leqslant \frac{\int_{\tilde M} |\nabla_{g} (f^*\phi)|^2 dv_{g}}{\int_{\tilde M} f^*\phi dv_{g}} = d \frac{\int_{\SSS} |\nabla_{g^*} \phi|^2 dv_{g^*}}{\int_{\SSS} \phi dv_{g^*}} = d\lambda_1^{\tau}(\SSS,g^*).
$$     
As a result, we get an inequality $\lambda_1^\sigma(\tilde M,g)\leqslant d\lambda_1^{\tau}(\SSS,g^*)$. Moreover using equality~(\ref{YY1}) for $u \equiv 1$ we obtain $Area_{g^*}(\SSS) = Area_g(\tilde M) = 2 Area_h(M)$. Finally, we have 
\begin{equation}
\label{bound}
\Lambda_1(M,h)\leqslant \frac{1}{2}d\lambda_1^\tau(\SSS,g^*) Area_{g^*}(\SSS).
\end{equation}

To derive Theorem~\ref{MainTheorem} from inequality~(\ref{bound}) we need to find upper bounds for $d$ (it is done in the following section) and for $\lambda_1^\tau(\SSS,g^*) Area_{g^*}(\SSS)$ (see Section~\ref{final}).

\section{Some facts from real algebraic geometry}
In this section we explain some terminology from algebraic geometry.
\label{algem}

In algebraic geometry $(\tilde M, \sigma)$ is called a real curve. The term "curve" is used due to the fact that $M$ has a structure of 1-dimensional complex manifold. The reason behind "real" is that such a curve can always be embedded in the projective space $\mathbb{CP}^n$ such that $\sigma$ becomes conjugation of the coordinates $(X_0:\ldots:X_n)\mapsto(\bar{X_0}:\ldots:\bar{X_n})$. Thus the image of the embedding is a zero locus of polynomials with real coefficients. The fact that $\sigma$ has no fixed points corresponds to the statement "$\tilde M$ has no real points" meaning that aforementioned embedding does not intersect the real locus $\mathbb{RP}^n\subset\mathbb{CP}^n$.

Then relation~(\ref{tau}) becomes a definition of a map between real curves $(\tilde M,\sigma)\to (\SSS,\tau)$. Any conformal automorphism of $\SSS$ inversing the orientation is conjugate to either $\tau_1(z) = \bar z$ or $\tau_2(z) = -\bar z^{-1}$. Thus without loss of generality we can assume that $\tau=\tau_1$ or $\tau=\tau_2$.

Let $d_1(\tilde M)$ be a minimal degree of the map $(\tilde M,\sigma)\to (\SSS,\tau_1)$. Then $d_1(\tilde M)$ coincides with the so called real gonality of $(\tilde M,\sigma)$, see e.g.~\cite{CM}. Let us also denote by $d_2(\tilde M)$ a minimal degree of the map $(\tilde M,\sigma)\to (\SSS,\tau_2)$.

We deduce Theorem~\ref{MainTheorem} from the following proposition.
\begin{proposition}
\label{prop}
Let $(M,h)$ be a non-orientable Riemannian surface and $(\tilde M, \sigma)$ is the corresponding double cover. Then one has the inequality
$$
\Lambda_1(M,h)\leqslant \min\{8\pi d_1(\tilde M), 12\pi d_2(\tilde M)\}.
$$ 
\end{proposition}
We prove this proposition in the following section.

Then Proposition~\ref{prop} together with the following lemma yields Theorem~\ref{MainTheorem}.
\begin{lemma}
If $\tilde M$ is a real curve of genus $\gamma$ then either $d_1(\tilde M)\leqslant 2\left[\frac{\gamma+3}{2}\right]$ or $d_2(\tilde M)\leqslant\left[\frac{\gamma+3}{2}\right]$
\label{lemma}
\end{lemma} 
\begin{proof}
By a well-known theorem, for any Riemann surface $\Sigma$ of genus $\gamma$ there exists a holomorphic map $f\colon\Sigma\to\bar{\mathbb{C}}$ of degree $d = 
\left[\frac{\gamma+3}{2}\right]$, where $\bar{\mathbb{C}}$ is a Riemann sphere. In other words there exists meromorphic function $f$ on $\Sigma$ such that generic point $z\in\mathbb{C}$ has exactly $d$ preimages. 

Consider a new function $F(z) = f(z)\bar{f}(\sigma(z))$. Note that $F(\sigma(z)) = \bar{F}(z)$. Therefore if $F\equiv C$ then $C\in \mathbb{R}$. Moreover $C\ne 0$ since $f$ is not identically zero. 

{\bf Case 1.}  $F(z) \equiv C>0$. Then by rescaling we can assume that $C=1$. Therefore the function $f$ satisfies relation~(\ref{tau}) for $\tau(z) = \frac{1}{\bar z}$. This transformation is conjugate to $\tau_1(z)$ by $z\mapsto\frac{z+i}{z-i}$. Thus in this case we have $d_1(\tilde M) = \left[\frac{\gamma+3}{2}\right]$. 

{\bf Case 2.} $F(z) \equiv C<0$. Again by rescaling we may assume $C=-1$. Therefore the function $f$ satisfies relation~(\ref{tau}) for $\tau(z) = -\frac{1}{\bar z}$, which coincides with $\tau_2(z)$. Thus in this case we have $d_2(\tilde M) = \left[\frac{\gamma+3}{2}\right]$.
 
{\bf Case 3.} $F(z)$ is a non-constant function. Then $F(z)$ is a meromorphic function of degree $2\left[\frac{\gamma+3}{2}\right]$ that satisfies relation~(\ref{tau}) for $\tau(z) = \bar z$, which is $\tau_1(z)$. Thus in this case we get $d_1(\tilde M) = 2\left[\frac{\gamma+3}{2}\right]$.
\end{proof}
\begin{remark} M. Coppens communicated to us that if $\gamma = 2^n-2$, then $\min\{d_1,d_2\}\leqslant \gamma/2 + 1$. According to him that can be proved in the same way as Theorem~$2$ in paper~\cite{CM}.
 Given that statement the inequality in Theorem~\ref{MainTheorem} becomes
$$
\Lambda_1(M,h)\leqslant 12\pi\left[\frac{\gamma+3}{2}\right],
$$
in case $\gamma = 2^n-2$ for some $n$.
\end{remark}

\section{Proof of Proposition~\ref{prop}}
Let $f\colon(\tilde M,\sigma) \to (\SSS, \tau_1)$ be a map of real curves of degree $d_1$, i.e. a meromorphic function of degree $d_1$ on $\tilde M$ that satisfies relation~(\ref{tau}) for $\tau(z) = \tau_1(z) = \bar z$.
\label{final}
 The fixed point set of $\tau_1$ is an equator in $\SSS$, therefore $\lambda_1^{\tau}(\SSS,g^*)$ can be identified with the first eigenvalue $\lambda_1(D,g^*)$ of a hemisphere $D$ with Neumann boundary conditions. Items 3 and 4 of Theorem~\ref{ThConf} imply that $\Lambda_1(D,g^*)\leqslant 8\pi$. Since the area of the hemisphere is half the volume of the sphere one has the inequality
$$
\lambda_1^{\tau}(\SSS,g^*) Area_{g^*}(\SSS)\leqslant 16\pi.
$$
Combining that with inequality~(\ref{bound}) one obtains 
$$
\Lambda_1(M,h)\leqslant 8\pi d_1.
$$ 

Let $f\colon(\tilde M,\sigma) \to (\SSS, \tau_2)$ be a map of real curves of degree $d_2$, i.e. a meromorphic function of degree $d_2$ on $\tilde M$ that satisfies relation~(\ref{tau}) for $\tau(z) = \tau_2(z) = -\bar z^{-1}$.
 In this case $\tau$ is antipodal involution and $\lambda_1^{\tau}(\SSS,g^*)$ can be identified with the first eigenvalue of $\RP = \SSS/\tau$ with the induced metric. Theorem~\ref{ThConf} yields the inequality 
$$
\Lambda_1(\RP,g)\leqslant 12\pi
$$
for any metric $g$. Proceeding as in the previous case we obtain the inequality
$$
\Lambda_1(M,h)\leqslant 12\pi d_2.
$$

\section{Hyperelliptic case}

In case $(\tilde M,g)$ is a hyperelliptic surface (e.g. $\tilde M$ is of genus $2$) it is possible to improve the statement of Lemma~\ref{lemma}.
\label{hyperelliptic}

\begin{lemma}
Suppose that $(\tilde M,g)$ is hyperelliptic of even genus $\gamma$. Then $d_1(\tilde M) = 2$.
\end{lemma}
\begin{proof}
Let $J$ denote the hyperelliptic involution. By Riemann-Hurwitz theorem $J$ has $2\gamma +2$ fixed points and therefore $J\ne\sigma$. Then by Corollary $3$ on p. 108 of book~\cite{FK}, $J$ commutes with every conformal diffeomorphism of $\tilde M$, in particular $J\circ \sigma = \sigma\circ J$. Therefore hyperelliptic prjection $\Pi$ satisfies  relation~(\ref{tau}) with some $\tau$. Without loss of generality we may assume $\tau = \tau_1$ or $\tau = \tau_2$. In case $\tau = \tau_2$ we obtain a branched 2-sheeted covering $M\to\RP$. Since $\gamma$ is even, the latter contradicts Hurwitz identities (see~\cite{EKS}). Thus, $\tau = \tau_1$ and $d_1(\tilde M) = 2$.
\end{proof}

Since any surface of genus two is hyperelliptic we obtain Theorem~\ref{gamma2} as a corollary.

\subsection*{Acknowledgements.} The author is grateful to I.~Polterovich for fruitful discussions and helpful remarks on the preliminary version of the manuscript, and to M.~Coppens for consultations on real algebraic geometry.

The research of the author was partially supported by Tomlinson Fellowship and Lorne Trottier Fellowship.

\end{document}